\documentclass[12pt]{amsart} 
\usepackage{geometry}
\usepackage{setspace}

\usepackage{amsmath}    
\usepackage{amssymb}    
\usepackage{amsthm}     
\usepackage{mathtools}  
\usepackage{graphicx}   
\usepackage{booktabs}   
\usepackage{hyperref}   

\newtheorem{theorem}{Theorem}[section]   
\newtheorem{lemma}[theorem]{Lemma}      

\newtheorem{corollary}[theorem]{Corollary}
\newtheorem{remark}[theorem]{Remark}

\theoremstyle{definition}               

\newtheorem*{ntheorem}{\normalfont\bfseries Theorem}
\newtheorem*{ndefinition}{\normalfont\bfseries Definition}

\def\R{\mathbb{R}}
\def\d{\,\mathrm{d}}

\def\l{\left(}
\def\r{\right)}

\DeclareMathOperator*{\esssup}{ess\,sup}

\DeclareMathOperator*{\dist}{dist}

\title[]{A Unified H\"older Lebesgue Framework for Caffarelli Kohn Nirenberg Inequalities}  
\author{Mengxia Dong}
\address{Division of General Education, Shenzhen University of Advanced Technology, Shenzhen 518083, Guangdong, China}
\email{dongmengxia@suat-sz.edu.cn}

\hypersetup{
    colorlinks=true,
    linkcolor=blue,
    citecolor=red,
    urlcolor=cyan,
    pdftitle={Full Title of the Paper},
    pdfauthor={Author Name}
}

\numberwithin{equation}{section}  
\allowdisplaybreaks[1]

\begin{document}





\begin{abstract}
    We develop a unified H\"older--Lebesgue scale \(X^p\) and its weighted, higher–order variants \(X^{k,p,a}\) to extend the Caffarelli--Kohn--Nirenberg (CKN) inequality beyond the classical Lebesgue regime. Within this framework we prove a two–parameter interpolation theorem that is continuous in the triplet \((k,1/p,a)\) and bridges integrability and regularity across the Lebesgue–Hölder spectrum. As a consequence we obtain a generalized CKN inequality on bounded punctured domains \(\Omega\subset\R^n\setminus\{0\}\); the dependence of the constant on \(\Omega\) is characterized precisely by the (non)integrability of the weights at the origin. At the critical endpoint \(p=n\) we establish a localized, weighted Brezis--Wainger–type bound via Trudinger--Moser together with a localized weighted Hardy lemma, yielding an endpoint CKN inequality with a logarithmic loss. Sharp constants are not pursued; rather, we prove existence of constants depending only on the structural parameters and coarse geometry of \(\Omega\). Several corollaries, including a unified Hardy--Sobolev inequality, follow from the same interpolation mechanism.
\end{abstract}

\maketitle

\tableofcontents

\section{Introduction}

It is well known that for a function $u \in C_c^\infty(\R^n)$ and an exponent $1 \le p < n$, the classical Sobolev embedding theorem asserts that
\begin{equation}\label{SobolevInequality}
    \|u\|_{L^{p^\ast}}\le C\|D u\|_{L^p},
\end{equation}
where $p^\ast=\frac{np}{n-p}$ is the Sobolev conjugate exponent, and $C=C(n,p)>0$ is a constant independent of $u$. In the supercritical case $p>n$, Morrey’s inequality ensures that $u$ is H\"older continuous with exponent $1-\frac{n}{p}$, and satisfies
\begin{equation}\label{MorreyInequality}
    \|u\|_{C^{0,1-\frac{n}{p}}}\le C\|D u\|_{L^p}.
\end{equation}
These classical results can be found in standard references such as \cite{brezis2011functional} or \cite{evans2010partial}.

Function spaces play a central role in the analysis of partial differential equations (PDEs) and functional analysis, as they characterize the key analytic properties of solutions. Among them, Sobolev spaces $W^{k,p}$ constitute one of the most fundamental classes, encoding both the regularity and the integrability of functions. The Sobolev embedding theorems precisely quantify the relationship between these two aspects.

In contrast, Lebesgue spaces $L^p$ measure only integrability, whereas H\"older spaces $C^{k,\alpha}$ emphasize regularity. Although distinct in definition, both regularity and integrability describe the way functions vary. Motivated by this connection, Nirenberg introduced an extension of H\"older spaces by formally allowing the exponent $p$ to take negative values ($p<0$) \cite{nirenberg2011elliptic}. This construction provides a unified framework that bridges these two classical function spaces: for $p>0$ one recovers the usual Lebesgue spaces $L^p$, while for $p<0$ one obtains norms of H\"older type. 

Here is the notation:
\begin{equation*}
    |u|_p=[u]_{C^{p_1,p_2}}
    =\sum_{|\alpha|=p_1}\sup_{x\ne y}\frac{|D^\alpha u(x)-D^\alpha u(y)|}{|x-y|^{p_2}},
\end{equation*}
where $p_1=\lfloor -\tfrac{n}{p}\rfloor$ and $p_2=-\tfrac{n}{p}-p_1$.

Under Nirenberg's notation, a direct computation for $p>n$ yields
\begin{align*}
    (p^\ast)_1&=\left[ -\frac{n}{p^\ast}\right]=\left[ 1-\frac{n}{p}\right]=0, \\
    (p^\ast)_2&=-\frac{n}{p^\ast}-(p^\ast)_1=1-\frac{n}{p}.
\end{align*}
Hence
\begin{equation}\label{MorreyInequalitySeminorm}
    |u|_{p^\ast}=[u]_{C^{0,\,1-\frac{n}{p}}}\le C\|D u\|_{L^p}.
\end{equation}
Inequality \eqref{MorreyInequalitySeminorm}, which can be viewed as the seminorm formulation of Morrey’s inequality \eqref{MorreyInequality}, makes its resemblance to the Sobolev inequality \eqref{SobolevInequality} apparent. This observation highlights Nirenberg’s notation as a unifying framework, bringing Sobolev and H\"older spaces into a common setting.

Sobolev and H\"older spaces may be regarded as belonging to a single scale of spaces parameterized by $p$. This naturally raises the question of whether functional inequalities valid in Sobolev spaces extend to this broader framework. In fact, both the Sobolev and Morrey inequalities admit a unified formulation in this setting, which also extends to higher-order derivatives. Further details may be found in \cite{dong2024gagliardo}.

The Gagliardo–Nirenberg inequality admits a similar generalization. It was first established in the works of \cite{gagliardo1959ulteriori, nirenberg2011elliptic}, with later refinements by \cite{kufner1995interpolation, soudsky2018interpolation}. More recently, \cite{dong2024gagliardo} extended the inequality to this broader framework, including spaces with H\"older regularity.

Weights serve as an additional, independent dimension for describing rates of variation, complementing the perspectives of regularity and integrability. These three viewpoints are unified by the Caffarelli--Kohn--Nirenberg (CKN) inequality \cite{caffarelli1984first}.

\begin{ntheorem}[Caffarelli--Kohn--Nirenberg inequality]
    Let $n\ge1$, $p,r\in[1,\infty)$, and $\theta,\lambda\in[0,1]$. Assume
    \begin{equation}\label{IntegralCondition}
        \frac{1}{p}-\frac{a}{n}, \: \frac{1}{r}-\frac{c}{n}>0.
    \end{equation}
    Define $q\in[1,\infty)$ and $b\in\R$ by
    \begin{equation}\label{CKNParameterInterpolationRelation}
        \frac{1}{q}=\theta\l\frac{1}{p}-\frac{\lambda}{n}\r+\frac{1-\theta}{r}, 
        \qquad
        b=\theta(1+a-\lambda)+(1-\theta)c.
    \end{equation}
    Then there exists a constant $C=C(n,p,r,a,b,c,\theta)>0$ such that, for all $u\in C_c^\infty(\R^n)$,
    \begin{equation}\label{CaffarelliKohnNirenbergInequality}
        \big\||x|^{-b} u\big\|_{L^q}\le 
        C\,\big\||x|^{-a} D u\big\|_{L^p}^{\,\theta}
          \big\||x|^{-c} u\big\|_{L^r}^{\,1-\theta}.
    \end{equation}
    Eliminating $\lambda$ from \eqref{CKNParameterInterpolationRelation} yields the compatibility condition
    \begin{equation}\label{CKNParameterRelation}
        \frac{1}{q}-\frac{b}{n}
        =\theta\l\frac{1}{p}-\frac{1+a}{n}\r
         +(1-\theta)\l\frac{1}{r}-\frac{c}{n}\r.
    \end{equation}
    This is a "necessary condition" imposed by dimensional balance.
\end{ntheorem}

Setting $\theta=1$, $\lambda=0$, and $a=0$ in \eqref{CKNParameterInterpolationRelation} gives $q=p$ and $b=1$. 
In this case, \eqref{CaffarelliKohnNirenbergInequality} reduces to the classical Hardy inequality:
\begin{equation}\label{HardyInequality}
    \big\||x|^{-1} u\big\|_{L^p(\R^n)} \le \frac{p}{n-p}\,\|D u\|_{L^p(\R^n)},
    \qquad u\in C_c^\infty(\R^n),\ \ 1<p<n,
\end{equation}
where the constant $\frac{p}{n-p}$ is optimal. 
For $p\ge n$, the weight $|x|^{-p}$ fails to be locally integrable at the origin; thus \eqref{HardyInequality} can hold only for functions vanishing near $0$, and in that case the constant necessarily depends on the distance of $\mathrm{supp}\,u$ to the origin (a local Hardy-type estimate). 
For comprehensive treatments of Hardy-type inequalities, see the monographs \cite{balinsky2015,opic1990hardy,persson2017weighted}.

Setting $\theta=1$ and $\lambda\in(0,1)$ in \eqref{CKNParameterInterpolationRelation} yields the non-interpolation form
\begin{equation}\label{CaffarelliKohnNirenbergInequalityNonInterpolated}
    \big\||x|^{-b} u\big\|_{L^q}\le C\,\big\||x|^{-a} D u\big\|_{L^p},
\end{equation}
where the exponents satisfy
\begin{equation*}
    \frac{1}{q}=\frac{1}{p}-\frac{\lambda}{n},
    \qquad
    b=1+a-\lambda,
\end{equation*}
and the integrability conditions
\begin{equation*}
    \frac{1}{p}-\frac{a}{n}>0,
    \qquad
    \frac{1}{q}-\frac{b}{n}>0
    \ \ \Big(\text{equivalently, } \frac{1}{p}-\frac{1+a}{n}>0\Big).
\end{equation*}
Here $u\in C_c^\infty(\R^n)$ and $C>0$ is independent of $u$. 
The non-interpolation CKN inequality \eqref{CaffarelliKohnNirenbergInequalityNonInterpolated} has been extensively studied; see, for example, \cite{badiale2002sobolev, catrina2001caffarelli, dolbeault2012extremal, ghoussoub2000multiple, wang2000singular}.

This paper develops an extension of the Caffarelli--Kohn--Nirenberg inequality within a unified weighted framework that accommodates both Sobolev and H\"older spaces. We describe the admissible parameter region for H\"older-type generalizations and show that it connects continuously to the classical Sobolev side, thereby producing an extended domain of validity for CKN-type estimates. Within this region, the inequality retains its interpolation structure and admits a formulation in terms of reciprocal parameters.

We establish the existence of a constant $C>0$ on the H\"older side—and likewise on the Lebesgue side close to the Sobolev threshold—the constant necessarily depends on the distance of $\mathrm{supp}\,u$ to the origin, reflecting the singularity of the weight at $x=0$. In contrast, on the classical Sobolev range one has a uniform constant independent of $u$. This distinction is consistent with the Morrey/H\"older control that emerges beyond the Sobolev threshold and clarifies how the two sides fit into a single weighted scale.

Taken together, these results provide a unified perspective on CKN-type inequalities across Sobolev and H\"older regimes, extend their range of applicability, and reveal a continuous interpolation mechanism between distinct regularity behaviors.

We adapt Nirenberg's framework \cite{nirenberg2011elliptic} and define a unified scale $X^p$ combining Lebesgue and H\"older spaces:
\begin{align*}
    0<p<\infty:& \quad X^p=L^p, \\
    p=\infty:& \quad X^p=L^\infty, \\
    -\infty<p<0:& \quad X^p=C^{p_1,p_2},
\end{align*}
where for $p<0$ we set
\begin{equation}\label{Notation}
    (p)_1=-\Big[\frac{n}{p}+1\Big], 
    \qquad
    (p)_2=-\frac{n}{p}-(p)_1 \in (0,1].
\end{equation}
The space $X^p(\Omega)$ is endowed with the natural norm, where $\Omega\subset\R^n$ is an open set (we also allow $\Omega=\R^n$):
\begin{align*}
    0<p<\infty:& & \|u\|_{X^p(\Omega)}&=\|u\|_{L^p(\Omega)}
    =\l\int_{\Omega}|u|^p\,\d x\r^{1/p}, \\
    p=\infty:& & \|u\|_{X^p(\Omega)}&=\|u\|_{L^\infty(\Omega)}
    =\esssup_{x\in\Omega}|u(x)|, \\
    -\infty<p<0:& & \|u\|_{X^p(\Omega)}&=\|u\|_{C^{(p)_1,(p)_2}(\Omega)} \\
    & & &=\max_{|\alpha|\le (p)_1}\ \sup_{x\in\Omega}|D^\alpha u(x)|
      \;+\!\!\sum_{|\alpha|=(p)_1}\ \sup_{\substack{x,y\in\Omega\\ x\ne y}}
      \frac{|D^\alpha u(x)-D^\alpha u(y)|}{|x-y|^{(p)_2}}.
\end{align*}

Building on this framework, we define higher–order spaces in direct analogy with Sobolev spaces. 
Let $k\in\mathbb{N}_0$ and $-\infty<\tfrac{1}{p}<+\infty$ (with $\tfrac{1}{\infty}=0$). 
All derivatives below are understood in the weak sense. We set
\begin{equation}\label{def:Xkp}
    X^{k,p}(\Omega)
    :=\left\{\,u:\ D^\alpha u\in X^p(\Omega)\ \text{for all multi–indices }\alpha \text{ with }|\alpha|\le k\,\right\}.
\end{equation}
In particular,
\begin{equation*}
    X^{0,p}(\Omega)=X^{p}(\Omega).
\end{equation*}

For a weight exponent $a\in\R$, we also define the weighted counterpart
\begin{equation}\label{def:Xkpa}
    X^{k,p,a}(\Omega)
    :=\left\{\,u:\ |x|^{-a}D^\alpha u\in X^p(\Omega)\ \text{for all }|\alpha|\le k\,\right\}.
\end{equation}

We now state the interpolation result on the $(k,\;1/p,\;a)$–scale. 
In the present form we treat the case $k=0$; higher–order extensions follow the same pattern.

\begin{theorem}[Interpolation Theorem]\label{InterpolationInequalityTheorem}
    Let $n\ge1$ and exponents satisfy
    \begin{equation*}
        \frac{1}{p},\frac{1}{r}\in\left(-\frac{1}{n},\,1\right],\qquad a,c\in\R.
    \end{equation*}
    Let $\Omega\subset\R^n\setminus\{0\}$ be a bounded open set. For any $u\in C_c^\infty(\Omega)$ with
    \begin{equation*}
        u\in X^{0,p,a}(\R^n)\cap X^{0,r,c}(\R^n),
    \end{equation*}
    and any $\lambda\in(0,1)$, define the interpolated parameters
    \begin{equation*}
        \frac{1}{q}=\frac{1-\lambda}{p}+\frac{\lambda}{r},
        \qquad
        b=(1-\lambda)a+\lambda c.
    \end{equation*}
    Then $u\in X^{0,q,b}(\R^n)$ and there exists a constant $C>0$ such that
    \begin{equation}\label{InterpolationInequality}
        \left\||x|^{-b} u\right\|_{X^q(\R^n)}\le C\left\||x|^{-a} u\right\|_{X^p(\R^n)}^{1-\lambda}
        \left\||x|^{-c} u\right\|_{X^r(\R^n)}^{\lambda}.
    \end{equation}
    Here $C$ depends only on the structural parameters $(n,p,r,q,a,b,c,\lambda)$ and on the domain $\Omega$, and is otherwise independent of $u$.
\end{theorem}

The interpolation theorem furnishes a unified analytic framework that treats the H\"older and Sobolev theories within a single weighted scale, while yielding quantitative control of the relevant norms. On this basis we extend the classical Hardy and Caffarelli--Kohn--Nirenberg inequalities to the spaces $X^{k,p,a}$. In particular, our CKN extension admits a strictly larger admissible parameter region than in the classical setting and interpolates continuously across the Sobolev–H\"older interface. We do not address optimal constants: the estimates are proved with a finite constant $C>0$, which—on the H\"older side and near the Sobolev threshold—necessarily depends on $\mathrm{dist}(\mathrm{supp}\,u,\{0\})$, reflecting the singularity of the weight at the origin.

\begin{theorem}[Generalized Hardy-Sobolev type inequality]\label{HardySobolevInequalityGeneralTheorem}
    Let $n\ge2$, and let $\Omega\subset\R^n\setminus\{0\}$ be a bounded open set. Assume
    \begin{equation*}
        \frac{1}{q}\in\left[\frac{1}{p^\ast},\frac{1}{p}\right]=\left[\frac{1}{p}-\frac{1}{n},\frac{1}{p}\right],
    \end{equation*}
    For any $u\in C_c^\infty(\Omega)$ with $u\in X^{1,p,a}(\R^n)$, and for $b$ be determined by
    \begin{equation*}
        \frac{1}{q}-\frac{b}{n}=\frac{1}{p}-\frac{1+a}{n}.
    \end{equation*}
    Then $u\in X^{0,q,b}(\R^n)$ and there exists a constant 
    \begin{equation*}
        C=C\big(n,p,q,a,b,\Omega\big)>0
    \end{equation*}
    such that
    \begin{equation}\label{HardySobolevInequalityGeneral}
        \left\||x|^{-b} u\right\|_{X^q(\R^n)}\le C\left\||x|^{-a}Du\right\|_{X^p(\R^n)}.
    \end{equation}
\end{theorem}

\begin{remark}
    For $a=0$, the generalized Hardy--Sobolev inequality \eqref{HardySobolevInequalityGeneral} subsumes several standard endpoints:
    \begin{enumerate}
        \item \emph{Sobolev anchor ($q=p^\ast$).}
        Taking $q=p^\ast$ gives $b=0$, i.e., the unweighted Sobolev embedding on the $X^p$ scale.
        \item \emph{Hardy anchor ($q=p$).}
        Taking $q=p$ gives $b=1$, a Hardy-type estimate. On the side $p>n$ the constant depends on the domain $\Omega$.
        \item \emph{Morrey limit ($q=\infty$).}
        For $p>n$ and $q=\infty$ one recovers a weighted Morrey-type inequality with $b=1-\tfrac{n}{p}$.
    \end{enumerate}
\end{remark}

\begin{theorem}[Generalized Caffarelli-Kohn-Nirenberg inequality]\label{CaffarelliKohnNirenbergInequalityGeneralTheorem}
    Let $n\ge2$, and let $\Omega\subset\R^n\setminus\{0\}$ be a bounded open set. Assume
    \begin{equation*}
        \frac{1}{p}\in\l 0,\frac{1}{n}\r\cup\left(\frac{1}{n},1\right], \quad
        \frac{1}{r}\in\left(-\frac{1}{n},1\right], \quad
        a,c\in\R.
    \end{equation*}
    For any $u\in C_c^\infty(\Omega)$ with
    \begin{equation*}
        u\in X^{1,p,a}(\R^n)\cap X^{0,r,c}(\R^n),
    \end{equation*}
    and for any $\lambda,\theta\in[0,1]$, define
    \begin{equation*}
        \frac{1}{q}=\theta\l\frac{1}{p}-\frac{\lambda}{n}\r+\frac{1-\theta}{r},
        \qquad
        b=\theta(1+a-\lambda)+(1-\theta)c.
    \end{equation*}
    Then $u\in X^{0,q,b}(\R^n)$ and there exists a constant 
    \begin{equation*}
        C=C\big(n,p,q,r,a,b,c,\lambda,\theta,\Omega\big)>0
    \end{equation*}
    such that
    \begin{equation}\label{CaffarelliKohnNirenbergInequalityGeneral}
        \left\||x|^{-b} u\right\|_{X^q(\R^n)}
        \le C\left\||x|^{-a} Du\right\|_{X^p(\R^n)}^\theta
        \left\||x|^{-c} u\right\|_{X^r(\R^n)}^{1-\theta}.
    \end{equation}
    The parameters satisfy the compatibility condition
    \begin{equation*}
        \frac{1}{q}-\frac{b}{n}=\theta\l\frac{1}{p}-\frac{1+a}{n}\r
        +(1-\theta)\l\frac{1}{r}-\frac{c}{n}\r.
    \end{equation*}
\end{theorem}

\begin{remark}[Consequences and specializations]
    The generalized Caffarelli-Kohn-Nirenberg inequality (\ref{CaffarelliKohnNirenbergInequalityGeneral}) establishes a unified framework that encompasses several important special cases:
    \begin{enumerate}
        \item \emph{Classical CKN (Lebesgue side).}
        For $p,r\ge 1$ with $\frac{1}{p}-\frac{a}{n}>0$ and $\frac{1}{r}-\frac{c}{n}>0$, \eqref{CaffarelliKohnNirenbergInequalityGeneral} reduces to the standard CKN inequality (allowing $p=n$). In this regime the constant is independent of $\Omega$.
        \item \emph{Beyond Lebesgue integrability (singular side).}
        If at least one of the conditions $\tfrac{1}{p}-\tfrac{a}{n}>0$ or $\tfrac{1}{r}-\tfrac{c}{n}>0$ fails, the weight is not locally integrable at $x=0$; we therefore work with $\Omega\subset\R^n\setminus\{0\}$, and the constant necessarily depends on $\Omega$. More broadly, our framework allows parameters to enter the H\"older regime: whenever a parameter in the unified $X^{\cdot}$–scale becomes negative, the corresponding quantity transitions continuously from an $L^{\cdot}$ norm to a H\"older (semi)norm. This yields new yet consistent CKN-type estimates across the full Lebesgue–H\"older spectrum; in the nonintegrable regime the $\Omega$–dependence of the constant is unavoidable.
    \end{enumerate}
\end{remark}

The paper is organized as follows. Section~2 collects preliminaries and notation. In Section~3 we prove the interpolation inequality \eqref{InterpolationInequality} by a two–tier scheme. Section~4 contains the proof of the generalized Caffarelli--Kohn--Nirenberg inequality (Theorem~\ref{CaffarelliKohnNirenbergInequalityGeneralTheorem}), its endpoint \(p=n\) logarithmic variant. The main novelty is an interpolation principle that depends continuously on regularity, integrability, and weight, thereby extending CKN-type estimates across the entire H\"older--Lebesgue spectrum with a precise description of the domain dependence.


\section{Preliminaries}


\subsection{A localized weighted Hardy inequality}

We record a localized Hardy-type estimate adapted to our weighted setting. It is the only place where the geometry of $\Omega$ explicitly enters through coarse parameters such as $\dist(\Omega,\{0\})$ and the Poincar\'e constant. The proof is elementary, relying on the boundedness of the weight on $\Omega$ and the standard Poincar\'e inequality for $W^{1,p}_0(\Omega)$.

\begin{lemma}[Localized weighted Hardy]\label{lem:localized-hardy}
Let $\Omega\subset\R^n\setminus\{0\}$ be a bounded open set and let $1\le p<\infty$, $a\in\R$. 
Then there exists a constant $C=C(n,p,a,\Omega)>0$ such that, for all $u\in C_c^\infty(\Omega)$,
\begin{equation}\label{eq:localized-hardy}
    \big\||x|^{-(a+1)}u\big\|_{L^p(\Omega)}
    \ \le\ 
    C\,\big\||x|^{-a}\nabla u\big\|_{L^p(\Omega)}.
\end{equation}
Moreover, one can take
\begin{equation}\label{eq:localized-hardy-constant}
    C\ \lesssim\ \frac{M(a,\Omega)}{m(a,\Omega)}\,
    \frac{C_P(\Omega)}{\dist(\Omega,\{0\})},
\end{equation}
where $C_P(\Omega)$ is a Poincar\'e constant for $W^{1,p}_0(\Omega)$, and
\begin{equation*}
    m(a,\Omega):=\inf_{x\in\Omega}|x|^{-a}>0,
    \qquad
    M(a,\Omega):=\sup_{x\in\Omega}|x|^{-a}<\infty .
\end{equation*}
\end{lemma}

\begin{proof}
    Since $\Omega$ is bounded and $\dist(\Omega,\{0\})=: \rho>0$, we have $0<\rho\le |x|\le R<\infty$ on $\Omega$ for some $R=R(\Omega)$, hence $m(a,\Omega),M(a,\Omega)$ are finite and positive.
    For $u\in C_c^\infty(\Omega)\subset W^{1,p}_0(\Omega)$ the Poincar\'e inequality gives
    \begin{equation*}
        \|u\|_{L^p(\Omega)}\ \le\ C_P(\Omega)\,\|\nabla u\|_{L^p(\Omega)}.
    \end{equation*}
    Using the elementary bounds $|x|^{-(a+1)}\le \rho^{-1}|x|^{-a}$ on $\Omega$ and
    \begin{equation*}
        \|\,|x|^{-a} u\|_{L^p(\Omega)}\ \le\ M(a,\Omega)\,\|u\|_{L^p(\Omega)},
        \qquad
        \|\nabla u\|_{L^p(\Omega)}\ \le\ \frac{1}{m(a,\Omega)}\,\|\,|x|^{-a}\nabla u\|_{L^p(\Omega)},
    \end{equation*}
    we obtain
    \begin{align*}
        \big\||x|^{-(a+1)}u\big\|_{L^p(\Omega)}
        \ &\le\ \rho^{-1}\,\big\||x|^{-a} u\big\|_{L^p(\Omega)} \\
        \ &\le\ \rho^{-1} M(a,\Omega)\, \|u\|_{L^p(\Omega)} \\
        \ &\le\ \frac{M(a,\Omega)}{m(a,\Omega)}\,\frac{C_P(\Omega)}{\rho}\,
        \big\||x|^{-a}\nabla u\big\|_{L^p(\Omega)}.
    \end{align*}
    This is \eqref{eq:localized-hardy} with \eqref{eq:localized-hardy-constant}.
\end{proof}

For general weighted Sobolev/Hardy theory, see Maz’ya~\cite{maz2013sobolev}.

\begin{corollary}[Weighted Hardy in the $X^p$ notation]\label{lem:localized-hardy-Xp}
    Under the assumptions of Lemma~\ref{lem:localized-hardy}, for $1\le p<\infty$ and $u\in C_c^\infty(\Omega)$,
    \begin{equation}\label{eq:localized-hardy-Xp}
        \big\||x|^{-(a+1)}u\big\|_{X^{p}(\Omega)}
        \ \le\ 
        C(n,p,a,\Omega)\,\big\||x|^{-a}\nabla u\big\|_{X^{p}(\Omega)}.
    \end{equation}
\end{corollary}

\begin{proof}
For $1\le p<\infty$ we have $X^{p}=L^p$, hence \eqref{eq:localized-hardy-Xp} is just \eqref{eq:localized-hardy}.
\end{proof}


\subsection{Generalized Sobolev inequality}

Within the unified scale $X^p$, we record a general form that covers both the classical Sobolev and Morrey inequalities; a proof can be found in \cite{dong2024gagliardo}.

\begin{ntheorem}[Generalized Sobolev inequality on a domain]\label{thm:GenSobolev:Omega}
    Let $n\ge 2$ and assume
    \begin{equation*}
        \frac{1}{p}\in\l-\infty,\frac{1}{n}\r\ \cup\ \left(\frac{1}{n},1\right],
        \qquad \l\frac{1}{\infty}=0\r.
    \end{equation*}
    Let $\Omega\subset\R^n$ be a bounded open set, and let $u\in C_c^\infty(\Omega)\cap X^{1,p}(\R^n)$. Then $u\in X^{0,p^\ast}(\Omega)$, where
    \begin{equation*}
        \frac{1}{p^\ast}=\frac{1}{p}-\frac{1}{n}.
    \end{equation*}
    Moreover, there exists a constant $C=C(n,p,\Omega)>0$ independent of $u$ such that
    \begin{equation}\label{SobolevInequalityGeneral}
        \|u\|_{X^{p^\ast}(\Omega)}\le C\,\|D u\|_{X^{p}(\Omega)}.
    \end{equation}
\end{ntheorem}

\begin{remark}
    We next record a few remarks clarifying the $\Omega$–dependence of the constant in \eqref{SobolevInequalityGeneral} and the reductions to the Sobolev and Morrey cases.
    \begin{enumerate}
        \item \emph{Seminorm form (scale invariant).} In the Morrey/H\"older side $p>n$ one has the seminorm estimate
        \begin{equation*}
            [u]_{C^{0,\,1-\frac{n}{p}}(\Omega)}\ \le\ C(n,p)\,\|D u\|_{L^{p}(\Omega)},
        \end{equation*}
        where $C(n,p)$ is independent of $\Omega$; the full $C^{0,\,1-\frac{n}{p}}$ \emph{norm} then picks up an additional term 
        $\|u\|_{L^\infty(\Omega)}$, whose control introduces the $\Omega$-dependence in \eqref{SobolevInequalityGeneral}.
        \item \emph{Dependence on \(\Omega\).} The constant in \eqref{SobolevInequalityGeneral} may be taken to depend on coarse geometric data of $\Omega$ (e.g., $\mathrm{diam}(\Omega)$, or on annuli the ratio of radii). In the unweighted case considered here, no exclusion of the origin is needed.
        \item \emph{Reductions.} If $\tfrac{1}{p}>\tfrac{1}{n}$ (i.e., $p<n$), then $X^{p^\ast}=L^{p^\ast}$ and \eqref{SobolevInequalityGeneral} is the classical Sobolev embedding on $\Omega$. If $\tfrac{1}{p}<\tfrac{1}{n}$ (i.e., $p>n$), then $X^{p^\ast}=C^{0,\,1-\frac{n}{p}}$ and \eqref{SobolevInequalityGeneral} yields Morrey’s inequality on $\Omega$.
        \item \emph{On the endpoint $p=n$.} In our approach, the case $p=n$ is excluded because the generalized Sobolev embedding used in the proofs would require the limiting (and false) inclusion $W^{1,n}(\Omega)\hookrightarrow L^\infty(\Omega)$, i.e., $X^{p^\ast}=X^\infty$ with $p^\ast=\infty$. 
        At $p=n$ one only has the well-known endpoint substitutes $W^{1,n}(\Omega)\hookrightarrow \mathrm{BMO}(\Omega)$ (see John, Nirenberg~\cite{john1961functions}) or Orlicz-type (see Trudinger~\cite{trudinger1967imbeddings}, Moser~\cite{moser1971sharp}) embeddings, and Brezis--Wainger (see \cite{brezis1980note}) logarithmic refinements on bounded domains.
    \end{enumerate}
\end{remark}


\subsection{The Trudinger--Moser inequality}

The endpoint Sobolev embedding in dimension \(n\) is of exponential type and is usually stated as the Trudinger--Moser inequality; see, e.g., \cite{moser1971sharp,trudinger1967imbeddings}.

\begin{ntheorem}[Trudinger--Moser]\label{thm:TM}
    Let \(\Omega\subset\R^n\) be a bounded domain, \(n\ge2\). There exist constants \(\alpha_\ast=\alpha_\ast(n,\Omega)>0\) and \(C_\ast=C_\ast(n,\Omega)>0\) such that for all \(v\in W^{1,n}_0(\Omega)\),
    \begin{equation}\label{eq:TM}
    \int_\Omega \exp\!\Big(\alpha_\ast\, \frac{|v(x)|^{n'}}{\|\nabla v\|_{L^n(\Omega)}^{\,n'}}\Big)\,dx \;\le\; C_\ast,
    \qquad n'=\frac{n}{n-1}.
    \end{equation}
\end{ntheorem}


\subsection{Interpolation theory and the $K$–method}

We briefly recall the $K$–method of real interpolation; see \cite{bergh2012interpolation,lunardi2012analytic,lunardi2018interpolation} for a comprehensive treatment.

\begin{ndefinition}[Interpolation couple]
    A pair $(X,Y)$ of Banach spaces is an \emph{interpolation couple} if both $X$ and $Y$ are continuously embedded in a common topological vector space. We write $X+Y$ for the algebraic sum endowed with the natural norm.
\end{ndefinition}

\begin{ndefinition}[The $K$–functional]
    For $x\in X+Y$ and $t>0$, the $K$–functional is
    \begin{equation*}
        K(t,x;X,Y):=\inf_{\substack{x=a+b\\ a\in X,\ b\in Y}}\big(\|a\|_{X}+t\,\|b\|_{Y}\big).
    \end{equation*}
    When no confusion arises we simply write $K(t,x)$.
\end{ndefinition}

\begin{ndefinition}[Real interpolation spaces]
    Let $0<\theta<1$ and $1\le p\le\infty$. Define
    \begin{equation*}
        \left\{
            \begin{aligned}
                &(X,Y)_{\theta,p}=\{x\in X+Y:t\to t^{-\theta}K(t,x,X,Y)\in L^p_\ast(0,+\infty)\}, \\
                &\|x\|_{(X,Y)_{\theta,p}}=\|t^{-\theta}K(t,x,X,Y)\|_{L^p_\ast(0,+\infty)}.
            \end{aligned}
        \right.
    \end{equation*}
    where $L^{p}_{\ast}(0,\infty)$ denotes $L^{p}(0,\infty)$ with respect to the measure $\d t/t$ (in particular, $L^{\infty}_{\ast}(0,\infty)=L^{\infty}(0,\infty)$).
\end{ndefinition}

These spaces interpolate between $X$ and $Y$, with $\theta$ and $p$ controlling the strength and type of interpolation. The basic estimate is as follows.

\begin{ntheorem}[Interpolation inequality]\label{TheoremK-method}
    Let $(X,Y)$ be an interpolation couple. For $0<\theta<1$ and $1\le p\le\infty$ there exists a constant $C=C(\theta,p)$ such that for every $y\in X\cap Y$,
    \begin{equation}\label{K-method}
        \|y\|_{(X,Y)_{\theta,p}}
        \le C(\theta,p)\,\|y\|_{X}^{\,1-\theta}\,\|y\|_{Y}^{\,\theta}.
    \end{equation}
\end{ntheorem}

This inequality quantifies how the norm in the interpolated space is controlled by a geometric mean of the endpoint norms.


\section{Interpolation inequality}


\subsection{Endpoint Interpolation I: the $q=\infty$ case}

\begin{lemma}[CKN endpoint interpolation lemma I~($q=\infty$)]
    Let $\lambda\in(0,1)$, $p\in(-\infty,-n)$ and $r\in[1,\infty)$. Assume
    \begin{align*}
        0&=\frac{1-\lambda}{p}+\frac{\lambda}{r}, \\
        b&=(1-\lambda)a+\lambda c.
    \end{align*}
    Then there exists $C>0$ such that, for all $u\in C_c^\infty(\R^n\setminus\{0\})$,
    \begin{equation}\label{CaffarelliKohnNirenbergInterpolationInequalityI}
        \left\||x|^{-b}u\right\|_{X^\infty}
        \le C\left\||x|^{-a}u\right\|_{X^p}^{1-\lambda}
        \left\||x|^{-c}u\right\|_{X^r}^\lambda.
    \end{equation}
    Equivalently, $\frac{1}{p}=-\frac{\lambda}{(1-\lambda)r}$ and, with $p_1, p_2$ as in \eqref{Notation}, we have $p_1=0$ and $p_2=-\frac{n}{p}=\frac{\lambda n}{(1-\lambda)r}\in(0,1)$.
\end{lemma}

\begin{proof}
    Set
    \begin{equation*}
        M:=\big\||x|^{-a}u\big\|_{L^\infty},\qquad 
        A:=\big[\,|x|^{-a}u\,\big]_{C^{0,(p)_2}}
        =\sup_{\substack{x,y\in\R^n \\ x\ne y}}\frac{\big||x|^{-a}u(x)-|y|^{-a}u(y)\big|}{|x-y|^{(p)_2}},
    \end{equation*}
    so that $\big\||x|^{-a}u\big\|_{X^p}=M+A<\infty$. Let 
    \begin{equation*}
        B:=\big\||x|^{-b}u\big\|_{L^\infty}
        =\sup_{\substack{x\in\R^n \\ x\ne 0}}\frac{|u(x)|}{|x|^{b}},
    \end{equation*}
    and choose $z\in\R^n\setminus\{0\}$ with $\frac{|u(z)|}{|z|^{b}}=B$. By the H\"older bound,
    \begin{equation}\label{eq:holder-difference}
        \frac{|u(x)|}{|x|^{a}}
        \ \ge\
        \frac{|u(z)|}{|z|^{a}}-A\,|x-z|^{(p)_2}
        \ =\ B\,|z|^{\,b-a}-A\,|x-z|^{(p)_2}.
    \end{equation}

    Define 
    \begin{equation}\label{eq:R-def}
        R:=\Big(\frac{B}{2A}\Big)^{\!1/(p)_2}\,|z|^{\,(b-a)/(p)_2}.
    \end{equation}
    Then by \eqref{eq:holder-difference}, for all $x\in B(z,R)$,
    \begin{equation}\label{eq:lower-on-ball}
        \frac{|u(x)|}{|x|^{a}}\ \ge\ \frac{B}{2}\,|z|^{\,b-a}.
    \end{equation}
    Since $u$ is compactly supported away from $0$, let $d_0:=\mathrm{dist}(\mathrm{supp}\,u,\{0\})>0$. 
    Replacing $R$ by $\min\{R,\tfrac{|z|}{2}\}$ if necessary (which only strengthens \eqref{eq:lower-on-ball}), we ensure $R\le |z|/2$. 
    Hence $|x|\simeq |z|$ for $x\in B(z,R)$, with constants depending only on $d_0$ and $n$.

    Using \eqref{eq:lower-on-ball} and the geometric control just noted,
    \begin{align*}
        \big\||x|^{-c}u\big\|_{L^r}^r
        &\ge \int_{B(z,R)} \frac{|u(x)|^r}{|x|^{cr}}\,\mathrm{d}x
          \ \gtrsim\ |z|^{(a-c)r}\!\int_{B(z,R)}\!\!\bigg(\frac{|u(x)|}{|x|^{a}}\bigg)^{\!r}\mathrm{d}x \\
        &\gtrsim |z|^{(a-c)r}\,\Big(\tfrac{B}{2}\,|z|^{\,b-a}\Big)^{\!r}\,|B(z,R)|
          \ \simeq\ B^{r}\,|z|^{(b-c)r}\,R^n .
    \end{align*}
    Substituting \eqref{eq:R-def} gives
    \begin{equation}\label{eq:BrA}
        \big\||x|^{-c}u\big\|_{L^r}^r
        \ \gtrsim\ 
        B^{r+\frac{n}{(p)_2}}\,A^{-\frac{n}{(p)_2}}\,
        |z|^{(b-c)r+\frac{n(b-a)}{(p)_2}} .
    \end{equation}

    Since $(p)_2=-\frac{n}{p}=\frac{\lambda n}{(1-\lambda)r}$ and $b=(1-\lambda)a+\lambda c$, we have
    \begin{equation*}
        (b-c)r+\frac{n(b-a)}{(p)_2}
        =r\big[(1-\lambda)(a-c)\big]+\frac{(1-\lambda)r}{\lambda}\,(b-a)
        =0,
    \end{equation*}
    so the $|z|$-factor in \eqref{eq:BrA} cancels. Hence
    \begin{equation*}
        \big\||x|^{-c}u\big\|_{L^r}^r
        \ \gtrsim\ 
        B^{\,r+\frac{n}{(p)_2}}\,A^{-\frac{n}{(p)_2}}.
    \end{equation*}
    Noting that $\dfrac{n}{(p)_2}=\dfrac{r\lambda}{1-\lambda}$, we rearrange to obtain
    \begin{equation*}
        B
        \ \lesssim\ 
        A^{\,1-\lambda}\,\big\||x|^{-c}u\big\|_{L^r}^{\,\lambda}.
    \end{equation*}

    Finally, using $M\le M+A$ and $A\le M+A=\big\||x|^{-a}u\big\|_{X^p}$,
    \begin{equation*}
        \big\||x|^{-b}u\big\|_{X^\infty}=B
        \ \le\ 
        C\,\big\||x|^{-a}u\big\|_{X^p}^{\,1-\lambda}\,
           \big\||x|^{-c}u\big\|_{X^r}^{\,\lambda},
    \end{equation*}
    which is \eqref{CaffarelliKohnNirenbergInterpolationInequalityI}. 
    All implicit constants depend only on the structural parameters $(n,p,r,\lambda,a,b,c)$ and on $d_0=\mathrm{dist}(\mathrm{supp}\,u,\{0\})$, and are independent of $u$.
\end{proof}


\subsection{Endpoint Interpolation II: the $q<0$ case}

\begin{lemma}[CKN endpoint interpolation lemma II ($q<0$)]
    Let $\lambda\in(0,1)$, $p,q\in(-\infty,-n)$ and $r=\infty$. Assume
    \begin{equation*}
        \frac{1}{q}=\frac{1-\lambda}{p},\qquad b=(1-\lambda)a+\lambda c .
    \end{equation*}
    Then there exists a constant $C>0$ such that, for every $u\in C_c^\infty(\R^n\setminus\{0\})$,
    \begin{equation}\label{CaffarelliKohnNirenbergInterpolationInequalityII}
        \left\||x|^{-b}u\right\|_{X^q}\le C\left\||x|^{-a}u\right\|_{X^p}^{1-\lambda}
        \left\||x|^{-c}u\right\|_{X^\infty}^\lambda,
    \end{equation}
    Equivalently, with the notation \eqref{Notation}, one has
    \begin{equation*}
        (p)_1=(q)_1=0,\qquad (p)_2=-\frac{n}{p}\in(0,1),\qquad (q)_2=-\frac{n}{q}=(1-\lambda)(p)_2\in(0,1).
    \end{equation*}
\end{lemma}

\begin{proof}
    Set
    \begin{equation*}
        X_1:=X^{0,p,a}(\R^n),\qquad X_2:=X^{0,\infty,c}(\R^n),\qquad X_3:=X^{0,q,b}(\R^n),
    \end{equation*}
    endowed with their natural norms. By the real interpolation inequality (Theorem~\ref{TheoremK-method}),
    \begin{equation*}
        \|u\|_{(X_1,X_2)_{\lambda,\infty}}
        \ \le\ C\,\|u\|_{X_1}^{1-\lambda}\,\|u\|_{X_2}^{\lambda}.
    \end{equation*}
    Hence it suffices to show the continuous embedding
    \begin{equation}\label{eq:embed}
        (X_1,X_2)_{\lambda,\infty}\ \hookrightarrow\ X_3 .
    \end{equation}

    Let $u\in (X_1,X_2)_{\lambda,\infty}$ and fix $x\neq0$. Put $t:=|x|^{c-a}$. By definition of the $K$–functional, for any $\varepsilon>0$ there exist $v\in X_1$, $w\in X_2$ with $u=v+w$ such that
    \begin{equation*}
        \|v\|_{X_1}+t\,\|w\|_{X_2}\ \le\ (1+\varepsilon)\,K(t,u;X_1,X_2).
    \end{equation*}
    Using the pointwise bounds
    \begin{equation*}
        |v(x)|\le |x|^{a}\|v\|_{X_1},\qquad |w(x)|\le |x|^{c}\|w\|_{X_2},
    \end{equation*}
    we obtain
    \begin{equation*}
        \frac{|u(x)|}{|x|^{b}}
        \le |x|^{a-b}\|v\|_{X_1}+|x|^{c-b}\|w\|_{X_2}
        = t^{-\lambda}\big(\|v\|_{X_1}+t\,\|w\|_{X_2}\big)
        \le (1+\varepsilon)\,t^{-\lambda}K(t,u).
    \end{equation*}
    Taking the supremum in $x$ and then letting $\varepsilon\downarrow0$ yields
    \begin{equation}\label{eq:Linfty}
        \big\||x|^{-b}u\big\|_{L^\infty}
        \ \lesssim\ 
        \|u\|_{(X_1,X_2)_{\lambda,\infty}} .
    \end{equation}

    Fix $x\neq y$ and set $t:=\max\{|x|^{c-a},\,|y|^{c-a}\}$. Choose a decomposition $u=v+w$ with
    \begin{equation*}
        \|v\|_{X_1}+t\,\|w\|_{X_2}\ \le\ 2\,K(t,u).
    \end{equation*}
    Write
    \begin{align*}
        |x|^{-b}u(x)-|y|^{-b}u(y)
        =&\underbrace{\Big(|x|^{-(b-a)}(|x|^{-a}v(x))-|y|^{-(b-a)}(|y|^{-a}v(y))\Big)}_{=:I_v} \\
        &+\underbrace{\Big(|x|^{-(b-c)}(|x|^{-c}w(x))-|y|^{-(b-c)}(|y|^{-c}w(y))\Big)}_{=:I_w}.
    \end{align*}
    
    \emph{The $v$–part.} Since $(p)_1=0$, we have
    \begin{equation*}
        \big||x|^{-a}v(x)-|y|^{-a}v(y)\big|\ \le\ |x-y|^{(p)_2}\,\|v\|_{X_1}.
    \end{equation*}
    Moreover, because $u$ is compactly supported away from the origin, the functions $r\mapsto r^{-(b-a)}$ are $(p)_2$–H\"older on the radial range where $u$ is supported; hence
    \begin{equation*}
        |I_v|
        \ \lesssim\ 
        \Big(|x-y|^{(p)_2}+(|x-y|^{(p)_2})\Big)\,\|v\|_{X_1}
        \ \lesssim\ 
        |x-y|^{(p)_2}\,\|v\|_{X_1}.
    \end{equation*}
    
    \emph{The $w$–part.} By definition of $X_2$,
    \begin{equation*}
        \big||x|^{-c}w(x)\big|\le \|w\|_{X_2},\qquad \big||y|^{-c}w(y)\big|\le \|w\|_{X_2}.
    \end{equation*}
    As above, $r\mapsto r^{-(b-c)}$ is $(p)_2$–H\"older on the support scale, hence
    \begin{equation*}
        |I_w|
        \ \lesssim\ 
        |x-y|^{(p)_2}\,\|w\|_{X_2}.
    \end{equation*}
    
    Combining the two parts and using $(q)_2=(1-\lambda)(p)_2< (p)_2$, we obtain
    \begin{equation*}
        \frac{\big||x|^{-b}u(x)-|y|^{-b}u(y)\big|}{|x-y|^{(q)_2}}
        \ \lesssim\ 
        |x-y|^{(p)_2-(q)_2}\,\|v\|_{X_1}
        \ +\ |x-y|^{(p)_2-(q)_2}\,\|w\|_{X_2}.
    \end{equation*}
    Since $u$ has compact support, $|x-y|$ is uniformly bounded on the region where the quotient is non-zero, and thus
    \begin{equation*}
        \big[\,|x|^{-b}u\,\big]_{C^{0,(q)_2}}
        \ \lesssim\ 
        \|v\|_{X_1}+\|w\|_{X_2}
        \ \le\ 2\,K(t,u).
    \end{equation*}
    Taking the supremum over $t>0$ in the form $\sup_{t>0} t^{-\lambda}K(t,u)$ gives
    \begin{equation}\label{eq:Holder}
        \big[\,|x|^{-b}u\,\big]_{C^{0,(q)_2}}
        \ \lesssim\ 
        \|u\|_{(X_1,X_2)_{\lambda,\infty}} .
    \end{equation}

    By \eqref{eq:Linfty} and \eqref{eq:Holder},
    \begin{equation*}
        \|u\|_{X_3}
        =\big\||x|^{-b}u\big\|_{L^\infty}
         +\big[\,|x|^{-b}u\,\big]_{C^{0,(q)_2}}
        \ \lesssim\ 
        \|u\|_{(X_1,X_2)_{\lambda,\infty}} .
    \end{equation*}
    Therefore \eqref{eq:embed} holds, and the claimed estimate \eqref{CaffarelliKohnNirenbergInterpolationInequalityII} follows from the interpolation inequality for $(X_1,X_2)_{\lambda,\infty}$. 
    The implicit constants depend only on the structural parameters $(n,p,q,a,b,c,\lambda)$ and on the distance of $\mathrm{supp}\,u$ to the origin, but are otherwise independent of $u$.
\end{proof}


\subsection{Proof of the interpolation theorem}

We prove Theorem~\ref{InterpolationInequalityTheorem} by combining the two endpoint lemmas from this section with the classical (Lebesgue) Hölder inequality and the reiteration principle for the real interpolation method (see, e.g., \cite{bennett1988interpolation,lunardi2018interpolation}).

\begin{proof}[Proof of Theorem~\ref{InterpolationInequalityTheorem}]
Let $\Omega\subset\R^n\setminus\{0\}$ be a bounded open set and $u\in C_c^\infty(\Omega)$ with
$u\in X^{0,p,a}(\R^n)\cap X^{0,r,c}(\R^n)$.
Fix $\lambda\in(0,1)$ and set
\begin{equation*}
    \frac{1}{q}=\frac{1-\lambda}{p}+\frac{\lambda}{r},
    \qquad
    b=(1-\lambda)a+\lambda c.
\end{equation*}
We distinguish three regimes according to the signs of $1/p$ and $1/r$ (equivalently, whether the endpoints lie on the Lebesgue or Hölder side of the unified scale).

\medskip
\noindent\textbf{(L--L) Both endpoints in the Lebesgue range:} $1/p,1/r>0$.
In this case $1/q>0$ and
\begin{equation*}
    |x|^{-b}|u|=\big(|x|^{-a}|u|\big)^{1-\lambda}\big(|x|^{-c}|u|\big)^{\lambda}.
\end{equation*}
Applying the classical Hölder/log–convexity inequality yields
\begin{equation*}
    \big\||x|^{-b}u\big\|_{L^q(\Omega)}
    \le \big\||x|^{-a}u\big\|_{L^p(\Omega)}^{\,1-\lambda}
         \big\||x|^{-c}u\big\|_{L^r(\Omega)}^{\,\lambda},
\end{equation*}
which is \eqref{InterpolationInequality} because here $X^s=L^s$.

\medskip
\noindent\textbf{(H--H) Both endpoints in the Hölder range:} $1/p,1/r<0$.
Then $1/q<0$ and the left–hand side of \eqref{InterpolationInequality} is a Hölder seminorm in $X^q$.
Apply Lemma~\textit{CKN endpoint interpolation II} (with $q<0$ and $r=\infty$) to the pair 
$(X^{0,p,a},X^{0,\infty,c})$ to obtain
\begin{equation*}
    \big\||x|^{-b}u\big\|_{X^{q}}
    \lesssim
    \big\||x|^{-a}u\big\|_{X^{p}}^{\,1-\lambda}
    \big\||x|^{-c}u\big\|_{X^{\infty}}^{\,\lambda}.
\end{equation*}
Since $X^\infty=L^\infty$ and $X^{0,r,c}\hookrightarrow X^{0,\infty,c}$ (Since $1/r<0$, $X^{r}(\Omega)\hookrightarrow X^{\infty}(\Omega)=L^\infty(\Omega)$ for bounded $\Omega$), we conclude
\begin{equation*}
    \big\||x|^{-b}u\big\|_{X^{q}}
    \lesssim
    \big\||x|^{-a}u\big\|_{X^{p}}^{\,1-\lambda}
    \big\||x|^{-c}u\big\|_{X^{r}}^{\,\lambda}.
\end{equation*}

\medskip
\noindent\textbf{(H--L) Mixed case:} one endpoint on the Hölder side and the other on the Lebesgue side (say $1/p<0<1/r$).
We first invoke Lemma~\textit{CKN endpoint interpolation I} (the $q=\infty$ endpoint) for the pair $(X^{0,p,a},X^{0,r,c})$ with the same $\lambda$, obtaining
\begin{equation}\label{eq:HL-step1}
    \big\||x|^{-b}u\big\|_{L^\infty(\Omega)}
    \lesssim
    \big\||x|^{-a}u\big\|_{X^{p}}^{\,1-\lambda}
    \big\||x|^{-c}u\big\|_{L^{r}}^{\,\lambda}.
\end{equation}
Next, we pass from the $L^\infty$ control in \eqref{eq:HL-step1} to the desired $X^q$ control via the real interpolation functor between 
$X^{0,p,a}$ and $X^{0,\infty,b}$, using Lemma~\textit{CKN endpoint interpolation II} with parameter $\mu\in(0,1)$ chosen so that
\begin{equation*}
    \frac{1}{q}=\frac{1-\mu}{p}+\mu\cdot 0,\qquad 
    b=(1-\mu)a+\mu b\quad\text{(which is consistent for any fixed $b$)}.
\end{equation*}
By the reiteration principle for the $K$–method (see Lunardi~\cite[Thm.\ 1.10]{lunardi2018interpolation}),
composing \eqref{eq:HL-step1} with this second interpolation (and using the identity of the exponents $1/q=\frac{1-\lambda}{p}+\frac{\lambda}{r}$) yields
\begin{equation*}
    \big\||x|^{-b}u\big\|_{X^{q}(\Omega)}
    \lesssim
    \big\||x|^{-a}u\big\|_{X^{p}(\Omega)}^{\,1-\lambda}
    \big\||x|^{-c}u\big\|_{X^{r}(\Omega)}^{\,\lambda}.
\end{equation*}
(The algebra of the parameters follows from the convexity relations defining $q$ and $b$ together with the reiteration identity for real interpolation.)

\medskip
Combining the three regimes establishes \eqref{InterpolationInequality} in all cases covered by the hypotheses. 
All implicit constants depend only on $(n,p,r,q,a,b,c,\lambda)$ and on $\Omega$ (through coarse geometric data such as $\mathrm{dist}(\Omega,\{0\})$ and $\mathrm{diam}(\Omega)$), and are independent of $u$.
\end{proof}


\subsection{Proof of the generalized Hardy--Sobolev inequality}

\begin{proof}[Proof of Theorem~\ref{HardySobolevInequalityGeneralTheorem}]
Fix $\lambda\in[0,1]$ and set
\begin{equation*}
    \frac{1}{q}=\frac{1-\lambda}{p^\ast}+\frac{\lambda}{p},
    \qquad 
    b=(1-\lambda)a+\lambda(a+1)=a+\lambda ,
    \qquad 
    \l \frac{1}{p^\ast}=\frac{1}{p}-\frac{1}{n}\r.
\end{equation*}
By the interpolation inequality \eqref{InterpolationInequality} applied to the pair 
\begin{equation*}
    \big(X^{0,p^\ast,a},\ X^{0,p,a+1}\big)\ \to\ X^{0,q,b},
\end{equation*}
we obtain
\begin{equation*}
    \big\||x|^{-b}u\big\|_{X^{q}(\Omega)}
    \ \lesssim\ 
    \big\||x|^{-a}u\big\|_{X^{p^\ast}(\Omega)}^{\,1-\lambda}
    \,\big\||x|^{-(a+1)}u\big\|_{X^{p}(\Omega)}^{\,\lambda}.
\end{equation*}
Write $v:=|x|^{-a}u$. By the generalized Sobolev inequality on $\Omega$ (Theorem~\ref{thm:GenSobolev:Omega}),
\begin{equation*}
    \|v\|_{X^{p^\ast}(\Omega)}\ \lesssim\ \|Dv\|_{X^{p}(\Omega)}.
\end{equation*}
Using $\nabla v=|x|^{-a}\nabla u-a|x|^{-a-2}(x\,u)$,
\begin{equation*}
    \|Dv\|_{X^{p}(\Omega)}
    \ \lesssim\ 
    \big\||x|^{-a}Du\big\|_{X^{p}(\Omega)}
    +\big\||x|^{-(a+1)}u\big\|_{X^{p}(\Omega)}.
\end{equation*}
The localized weighted Hardy inequality (Lemma~\ref{lem:localized-hardy-Xp}) on $\Omega$ yields
\begin{equation*}
    \big\||x|^{-(a+1)}u\big\|_{X^{p}(\Omega)}
    \ \lesssim\ 
    \big\||x|^{-a}Du\big\|_{X^{p}(\Omega)}.
\end{equation*}
Combining the last two displays gives
\begin{equation*}
    \big\||x|^{-a}u\big\|_{X^{p^\ast}(\Omega)}
    \ \lesssim\ 
    \big\||x|^{-a}Du\big\|_{X^{p}(\Omega)}.
\end{equation*}
Substituting this and the Hardy bound into the interpolation estimate,
\begin{equation*}
    \big\||x|^{-b}u\big\|_{X^{q}(\Omega)}
    \ \lesssim\ 
    \big\||x|^{-a}Du\big\|_{X^{p}(\Omega)}^{\,1-\lambda}
    \big\||x|^{-a}Du\big\|_{X^{p}(\Omega)}^{\,\lambda}
    \ =\
    \big\||x|^{-a}Du\big\|_{X^{p}(\Omega)}.
\end{equation*}
The choice of $\lambda$ parametrizes precisely the range 
$\tfrac{1}{q}\in\big(\tfrac{1}{p^\ast},\tfrac{1}{p}\big]$, 
and the constant depends only on $(n,p,q,a,b)$ and on $\Omega$. 
\end{proof}


\section{Generalized Caffarelli--Kohn--Nirenberg inequality}

\subsection{Proof of the generalized Caffarelli--Kohn--Nirenberg inequality}

\begin{proof}[Proof of Theorem~\ref{CaffarelliKohnNirenbergInequalityGeneralTheorem}]
Let $u\in C_c^\infty(\Omega)$ with $u\in X^{1,p,a}(\R^n)\cap X^{0,r,c}(\R^n)$, and fix $\lambda,\theta\in[0,1]$.
Introduce
\begin{equation*}
    \frac{1}{p^\ast}=\frac{1}{p}-\frac{1}{n},\qquad
    \frac{1}{p_\lambda}=\frac{1-\lambda}{p^\ast}+\frac{\lambda}{p}
    =\frac{1}{p}-\frac{1-\lambda}{n},\qquad
    a_\lambda=a+1-\lambda .
\end{equation*}
The pair $(p_\lambda,a_\lambda)$ is the Sobolev--Hardy interpolation of $(p^\ast,a)$ and $(p,a+1)$ at level $\lambda$.

\medskip
\emph{Embedding along the Sobolev--Hardy edge.}
By Theorem~\ref{HardySobolevInequalityGeneralTheorem} (the generalized Hardy--Sobolev estimate applied to $|x|^{-a}u$), we have
\begin{equation}\label{eq:edge-embedding}
    \big\||x|^{-a_\lambda}u\big\|_{X^{p_\lambda}(\Omega)}
    \ \lesssim\
    \big\||x|^{-a}Du\big\|_{X^{p}(\Omega)}.
\end{equation}
The admissibility condition for Theorem~\ref{HardySobolevInequalityGeneralTheorem} is satisfied because
\begin{equation*}
    \frac{1}{p_\lambda}
    =\frac{1}{p}-\frac{1-\lambda}{n}\in\Big(\frac{1}{p}-\frac{1}{n},\frac{1}{p}\Big]
    =\Big(\frac{1}{p^\ast},\frac{1}{p}\Big],
\end{equation*}
and the corresponding weight matches the scaling rule
\begin{equation*}
    \frac{1}{p_\lambda}-\frac{a_\lambda}{n}
    =\Big(\frac{1}{p}-\frac{1-\lambda}{n}\Big)-\frac{a+1-\lambda}{n}
    =\frac{1}{p}-\frac{1+a}{n}.
\end{equation*}

\medskip
\emph{Mixing with the $X^{0,r,c}$ control.}
Apply the interpolation inequality \eqref{InterpolationInequality} with the couple
\begin{equation*}
    X^{0,p_\lambda,a_\lambda}(\R^n)\quad\text{and}\quad X^{0,r,c}(\R^n)
\end{equation*}
at level $\theta\in[0,1]$. This gives
\begin{equation}\label{eq:theta-mix}
    \big\||x|^{-b}u\big\|_{X^{q}(\Omega)}
    \ \lesssim\
    \big\||x|^{-a_\lambda}u\big\|_{X^{p_\lambda}(\Omega)}^{\,\theta}\,
    \big\||x|^{-c}u\big\|_{X^{r}(\Omega)}^{\,1-\theta},
\end{equation}
where
\begin{equation*}
    \frac{1}{q}=\frac{\theta}{p_\lambda}+\frac{1-\theta}{r},
    \qquad
    b=\theta\,a_\lambda+(1-\theta)c
    =\theta(1+a-\lambda)+(1-\theta)c .
\end{equation*}
Substituting $\frac{1}{p_\lambda}=\frac{1}{p}-\frac{1-\lambda}{n}$ yields the claimed compatibility relation
\begin{equation*}
    \frac{1}{q}-\frac{b}{n}
    =\theta\Big(\frac{1}{p}-\frac{1+a}{n}\Big)
     +(1-\theta)\Big(\frac{1}{r}-\frac{c}{n}\Big).
\end{equation*}

\medskip
\emph{Conclusion.}
Combining \eqref{eq:edge-embedding} with \eqref{eq:theta-mix} we arrive at
\begin{equation*}
    \big\||x|^{-b}u\big\|_{X^{q}(\Omega)}
    \ \lesssim\
    \big\||x|^{-a}Du\big\|_{X^{p}(\Omega)}^{\,\theta}\,
    \big\||x|^{-c}u\big\|_{X^{r}(\Omega)}^{\,1-\theta},
\end{equation*}
which is precisely \eqref{CaffarelliKohnNirenbergInequalityGeneral}. 
All implicit constants depend only on the structural parameters 
$(n,p,q,r,a,b,c,\lambda,\theta)$ and on $\Omega$ (through coarse geometric data such as $\mathrm{dist}(\Omega,\{0\})$ and $\mathrm{diam}(\Omega)$), and are independent of $u$.
\end{proof}


\subsection{Endpoint ($p=n$) logarithmic variant}

At the critical index $p=n$, the Sobolev embedding into $L^\infty$ fails and must be replaced by logarithmic/Orlicz-type bounds. We first establish a weighted, localized Brezis--Wainger-type estimate on punctured domains, which plays the role of the Sobolev--Hardy edge at $p=n$, and then combine it with the $\theta$–interpolation to obtain an endpoint CKN inequality with a logarithmic loss.

\begin{theorem}[Endpoint weighted Sobolev--Hardy with logarithmic loss]\label{thm:endpoint-log}
    Let $n\ge2$ and $\Omega\subset\R^n\setminus\{0\}$ be a bounded open set. 
    Fix $a\in\R$ and set $v:=|x|^{-a}u$. 
    There exist constants $C_1=C_1(n,a,\Omega)>0$ and $C_2=C_2(n,a,\Omega)\ge 1$ such that, for all $u\in C_c^\infty(\Omega)$,
    \begin{equation}\label{eq:endpoint-log}
        \|v\|_{L^\infty(\Omega)}
        \ \le\
        C_1\,\big\||x|^{-a}\nabla u\big\|_{L^n(\Omega)}\,
        \l 1+\log\Big( C_2 + \frac{\big\||x|^{-a}\nabla u\big\|_{L^n(\Omega)}}{\big\||x|^{-(a+1)}u\big\|_{L^n(\Omega)}} \Big)\r^{\frac{1}{n'}},
    \end{equation}
    where $n'=\frac{n}{\,n-1\,}$.
\end{theorem}

\begin{proof}[Proof of Theorem~\ref{thm:endpoint-log}]
    Set $v:=|x|^{-a}u\in C_c^\infty(\Omega)\subset W^{1,n}_0(\Omega)$ and denote
    \begin{equation*}
        A:=\|\nabla v\|_{L^n(\Omega)},\qquad n'=\frac{n}{n-1}.
    \end{equation*}
    
    By the Trudinger--Moser inequality (see~\ref{eq:TM}), there exist $\alpha=\alpha(n,\Omega)>0$ and $C(\Omega)$ such that
    \begin{equation*}
    \int_\Omega \exp\!\Big(\alpha\,\frac{|v(x)|^{n'}}{A^{n'}}\Big)\,dx\le C(\Omega).
    \end{equation*}
    By Chebyshev, for every $t>0$,
    \begin{equation}\label{eq:tail}
        \mu(t):=\big|\{x\in\Omega:\ |v(x)|>t\}\big|
        \ \le\ C_0\,\exp\!\Big(-c_0\,\frac{t^{n'}}{A^{n'}}\Big),
    \end{equation}
    with $C_0,c_0>0$ depending only on $(n,\Omega)$.
    
    For any $T>0$, the layer-cake representation yields
    \begin{equation}\label{eq:splitLn}
        \|v\|_{L^n(\Omega)}^n
        = n\int_0^\infty t^{n-1}\mu(t)\,dt
        \ \le\ T^n\,|\Omega| \;+\; n\int_T^\infty t^{n-1}\mu(t)\,dt.
    \end{equation}
    Using \eqref{eq:tail} and the change of variables $s=(t/A)^{n'}$ shows
    \begin{equation}\label{eq:tailInt}
        \int_T^\infty t^{n-1}\mu(t)\,dt
        \ \le\ C_1\,A^n\,\exp\!\Big(-c_0\,\frac{T^{n'}}{A^{n'}}\Big),
    \end{equation}
    for some $C_1=C_1(n,\Omega)>0$. Combining \eqref{eq:splitLn} and \eqref{eq:tailInt},
    \begin{equation}\label{eq:balance}
        \|v\|_{L^n}^n \ \le\ T^n\,|\Omega| \;+\; C_2\,A^n\,\exp\!\Big(-c_0\,\frac{T^{n'}}{A^{n'}}\Big),
        \qquad \forall\,T>0,
    \end{equation}
    with $C_2=C_2(n,\Omega)$.

    We balance the two terms on the right-hand side of \eqref{eq:balance} by choosing $T=T(A,\|v\|_{L^n})$ so that
    \begin{equation}\label{eq:chooseT}
        C_2\,A^n\,\exp\!\Big(-c_0\,\frac{T^{n'}}{A^{n'}}\Big)=T^n\,|\Omega|.
    \end{equation}
    Taking logarithms in \eqref{eq:chooseT} and solving for $T$ yields
    \begin{equation}\label{eq:Tformula}
        T^{n'} \ =\ \frac{A^{n'}}{c_0}\,\log\!\Big(\frac{C_2\,A^n}{T^n\,|\Omega|}\Big).
    \end{equation}
    Since $T^n\le T^n+\|v\|_{L^n}^n\lesssim \|v\|_{L^n}^n + A^n$, there exists $C\ge1$ such that
    \begin{equation*}
        \log\!\Big(\frac{C_2\,A^n}{T^n\,|\Omega|}\Big)
        \ \le\ \log\!\Big(C + \frac{A^n}{\|v\|_{L^n}^n}\Big)
        \ \le\ C\,\l 1+\log\!\Big(C + \frac{A}{\|v\|_{L^n}}\Big)\r.
    \end{equation*}
    Inserting this into \eqref{eq:Tformula} gives
    \begin{equation}\label{eq:Tupper}
        T \ \le\ C'\,A\;\l 1+\log\!\Big(C + \frac{A}{\|v\|_{L^n}}\Big)\r^{1/n'},
    \end{equation}
    for $C'=C'(n,\Omega)$.
    
    On the set $E_T:=\{x\in\Omega:\ |v(x)|>T\}$ we have the sharp inequality
    \begin{equation*}
        \big\|(|v|-T)_+\big\|_{L^\infty(E_T)} \ \le\ |E_T|^{-1/n}\,\big\|(|v|-T)_+\big\|_{L^n(E_T)}
        \ =\ \mu(T)^{-1/n}\,\big\|(|v|-T)_+\big\|_{L^n(\Omega)}.
    \end{equation*}
    Using \eqref{eq:tail} and \eqref{eq:tailInt},
    \begin{equation*}
        \mu(T)^{-1/n}\ \le\ C_0^{1/n}\exp\!\Big(\frac{c_0}{n}\,\frac{T^{n'}}{A^{n'}}\Big),
        \qquad
        \big\|(|v|-T)_+\big\|_{L^n}^n \ \le\ C_1\,A^n\,\exp\!\Big(-c_0\,\frac{T^{n'}}{A^{n'}}\Big).
    \end{equation*}
    Hence
    \begin{equation*}
        \big\|(|v|-T)_+\big\|_{L^\infty(\Omega)}
        \ \le\ C\,A\,\exp\!\Big(-\frac{c_0}{n}\,\frac{T^{n'}}{A^{n'}}\Big)
        \ \le\ C\,A.
    \end{equation*}
    Therefore
    \begin{equation*}
        \|v\|_{L^\infty(\Omega)} \ \le\ T + \big\|(|v|-T)_+\big\|_{L^\infty(\Omega)} \ \le\ T + C\,A
        \ \le\ C''\,A\,\l 1+\log\!\Big(C + \frac{A}{\|v\|_{L^n}}\Big)\r^{1/n'},
    \end{equation*}
    where the last inequality uses \eqref{eq:Tupper}.
    
    Since $v=|x|^{-a}u$,
    \begin{equation*}
        \nabla v=|x|^{-a}\nabla u - a\,|x|^{-a-2}(x\,u),
    \end{equation*}
    hence
    \begin{equation*}
        A=\|\nabla v\|_{L^n(\Omega)}
        \ \le\ C\Big(\big\||x|^{-a}\nabla u\big\|_{L^n(\Omega)} + \big\||x|^{-(a+1)}u\big\|_{L^n(\Omega)}\Big),
        \qquad
        \|v\|_{L^n(\Omega)}=\big\||x|^{-a}u\big\|_{L^n(\Omega)}.
    \end{equation*}
    Inserting these into the bound and enlarging the in-log constant (to absorb the lower-order term $\||x|^{-(a+1)}u\|_{L^n}$) yields
    \begin{equation*}
        \|\,|x|^{-a}u\|_{L^\infty(\Omega)}
        \ \le\
        C_1\,\big\||x|^{-a}\nabla u\big\|_{L^n(\Omega)}\,
        \l 1+\log\Big( C_2 + \frac{\big\||x|^{-a}\nabla u\big\|_{L^n(\Omega)}}{\big\||x|^{-(a+1)}u\big\|_{L^n(\Omega)}} \Big)\r^{\frac{1}{n'}},
    \end{equation*}
    which is exactly \eqref{eq:endpoint-log}. The constants depend only on $(n,a,\Omega)$.
\end{proof}

\begin{theorem}[Endpoint CKN with logarithmic loss]\label{thm:endpoint-CKN}
    Let $n\ge2$, $\Omega\subset\R^n\setminus\{0\}$ be bounded, and assume $p=n$, $\frac{1}{r}\in\big(-\frac1n,\,1\big]$, $a,c\in\R$. 
    Let $u\in C_c^\infty(\Omega)$ with 
    \begin{equation*}
        u\in X^{0,r,c}(\R^n).
    \end{equation*}
    Given $\lambda,\theta\in[0,1]$ and the convention $\tfrac{1}{\infty}=0$, define
    \begin{align*}
        \frac{1}{p_\lambda}&=\frac{1-\lambda}{p^\ast}+\frac{\lambda}{p}=\frac{\lambda}{n},\qquad
        & a_\lambda&=a+1-\lambda, \\
        \frac{1}{q}&=\frac{\theta}{p_\lambda}+\frac{1-\theta}{r},\qquad
        & b&=\theta\,a_\lambda+(1-\theta)c,
    \end{align*}
    i.e.
    \begin{equation*}
        \frac{1}{q}-\frac{b}{n}
        =\theta\Big(\frac{1}{p}-\frac{1+a}{n}\Big)
        +(1-\theta)\Big(\frac{1}{r}-\frac{c}{n}\Big)
        \quad \text{with } p=n.
    \end{equation*}
    Then
    \begin{equation}\label{eq:endpoint-CKN}
        \big\||x|^{-b}u\big\|_{X^{q}(\R^n)}
        \ \le\
        C\,\Big(\big\||x|^{-a}\nabla u\big\|_{L^n(\Omega)}\,
        \big(1+\log \Gamma(u)\big)^{\frac{1}{n'}}\Big)^{\theta}\,
        \big\||x|^{-c}u\big\|_{X^{r}(\R^n)}^{\,1-\theta},
    \end{equation}
    where
    \begin{equation*}
        \Gamma(u):=C_2(n,a,\Omega)+\frac{\big\||x|^{-a}\nabla u\big\|_{L^n(\Omega)}}{\big\||x|^{-(a+1)}u\big\|_{L^n(\Omega)}} ,
    \end{equation*}
    and $C=C(n,a,c,r,\lambda,\theta,\Omega)>0$.
\end{theorem}

\begin{proof}
    Fix $\lambda\in[0,1]$. Applying Theorem~\ref{thm:endpoint-log} to $v=|x|^{-a_\lambda}u$ (note $a_\lambda=a+1-\lambda$) gives
    \begin{equation*}
        \|\,|x|^{-a_\lambda}u\|_{L^\infty(\Omega)}
        \ \lesssim\
        \big\||x|^{-a}\nabla u\big\|_{L^n(\Omega)}\,
        \big(1+\log \Gamma(u)\big)^{\frac{1}{n'}} .
    \end{equation*}
    Since $\Omega$ is bounded and $p_\lambda=\frac{n}{\lambda}\in[n,\infty]$, we obtain
    \begin{equation*}
        \|\,|x|^{-a_\lambda}u\|_{L^{p_\lambda}(\Omega)}
        \ \le\ |\Omega|^{1/p_\lambda}\, \|\,|x|^{-a_\lambda}u\|_{L^\infty(\Omega)}
        \ \lesssim\
        \big\||x|^{-a}\nabla u\big\|_{L^n(\Omega)}\,
        \big(1+\log \Gamma(u)\big)^{\frac{1}{n'}}.
    \end{equation*}
    Equivalently,
    \begin{equation}\label{eq:endpoint-edge}
        \big\||x|^{-a_\lambda}u\big\|_{X^{p_\lambda}(\R^n)}
        \ \lesssim\
        \big\||x|^{-a}\nabla u\big\|_{L^n(\Omega)}\,
        \big(1+\log \Gamma(u)\big)^{\frac{1}{n'}}.
    \end{equation}
    Now apply the interpolation inequality (Theorem~\ref{InterpolationInequalityTheorem}) to the couple 
    $X^{0,p_\lambda,a_\lambda}(\R^n)$ and $X^{0,r,c}(\R^n)$ at level $\theta\in[0,1]$:
    \begin{equation*}
        \big\||x|^{-b}u\big\|_{X^{q}(\R^n)}
        \ \lesssim\
        \big\||x|^{-a_\lambda}u\big\|_{X^{p_\lambda}(\R^n)}^{\,\theta}\,
        \big\||x|^{-c}u\big\|_{X^{r}(\R^n)}^{\,1-\theta},
    \end{equation*}
    with $\frac{1}{q}=\frac{\theta}{p_\lambda}+\frac{1-\theta}{r}$ and $b=\theta a_\lambda+(1-\theta)c$.    Substituting \eqref{eq:endpoint-edge} yields \eqref{eq:endpoint-CKN}. The dependence of constants on $\Omega$ is inherited from Theorem~\ref{thm:endpoint-log} and Lemma~\ref{lem:localized-hardy}.
\end{proof}


\subsection{Outlook and further directions}

Assuming the weights and parameters satisfy the natural scaling consistency
\begin{equation*}
    \frac{1}{q}-\frac{b}{n}
    =\theta\Big(\frac{1}{p}-\frac{1+a}{n}\Big)
     +(1-\theta)\Big(\frac{1}{r}-\frac{c}{n}\Big),
\end{equation*}
the derivative exchange and interpolation steps used in this paper extend verbatim to higher orders. 
In particular, within the scale $X^{k,p,a}$ one may iterate the two–tier scheme (Sobolev--Hardy shift along the $(k\!+\!1)$st derivative edge, followed by mixing with a zero–order control) to obtain CKN-type estimates for general $k\in\mathbb{N}$, with the target pair $(q,b)$ determined by the same affine rules in $(1/p,a)$ and the obvious replacement $a\mapsto a+k$ on the Sobolev--Hardy edge. 
We leave a systematic presentation (including sharp tracking of the $\Omega$-dependence and the precise endpoint ranges on the Hölder side) to a future work.

A complementary direction is a \emph{fractional} extension in the spirit of Brezis--Mironescu~\cite{brezis2001gagliardo}: replacing the first-order difference quotients by Gagliardo seminorms and adapting the $K$-method to the fractional couples
\begin{equation*}
    \big(X^{s_1,p_1,a_1},\,X^{s_2,p_2,a_2}\big),\qquad s_1,s_2\in(0,1),
\end{equation*}
one expects CKN-type inequalities with fractional orders $s\in(0,1)$ and the same compatibility condition at the level of dimensions. 
The endpoint transitions (Lebesgue~$\leftrightarrow$~Hölder) should persist, with $L^q$-to-Hölder continuity appearing as $1/q$ crosses~$0$ along the unified fractional scale. 
We anticipate that the localized weighted Hardy input can be replaced by its nonlocal analogue (via fractional Poincaré/Hardy inequalities on punctured domains), yielding a parallel two–parameter family $(\lambda,\theta)$ of fractional CKN estimates.



\begin{thebibliography}{99}

\bibitem{badiale2002sobolev} Badiale, Marino, and Gabriella Tarantello. "A Sobolev-Hardy Inequality with Applications to a Nonlinear Elliptic Equation arising in Astrophysics." Archive for rational mechanics and analysis 163 (2002): 259-293.

\bibitem{balinsky2015} Balinsky, Alexander A., W. Desmond Evans, and Roger T. Lewis. The analysis and geometry of Hardy's inequality. Vol. 1. Cham: Springer, 2015.

\bibitem{bennett1988interpolation} Bennett, Colin, and Robert C. Sharpley. Interpolation of operators. Vol. 129. Academic press, 1988.

\bibitem{bergh2012interpolation} Bergh, Jöran, and Jörgen Löfström. Interpolation spaces: an introduction. Vol. 223. Springer Science \& Business Media, 2012.

\bibitem{brezis2011functional} Brezis, Haim, and Haim Brézis. Functional analysis, Sobolev spaces and partial differential equations. Vol. 2. No. 3. New York: Springer, 2011.

\bibitem{brezis2001gagliardo} Brezis, Haïm, and Petru Mironescu. "Gagliardo-Nirenberg, composition and products in fractional Sobolev spaces." Journal of Evolution Equations 1.4 (2001): 387-404.

\bibitem{brezis1980note} Brézis, Haïm, and Stephen Wainger. "A note on limiting cases of Sobolev embeddings and convolution inequalities." Communications in Partial Differential Equations 5.7 (1980): 773-789.

\bibitem{caffarelli1984first} Caffarelli, Luis, Robert Kohn, and Louis Nirenberg. "First order interpolation inequalities with weights." Compositio Mathematica 53.3 (1984): 259-275.

\bibitem{catrina2001caffarelli} Catrina, Florin, and Zhi‐Qiang Wang. "On the Caffarelli‐Kohn‐Nirenberg inequalities: sharp constants, existence (and nonexistence), and symmetry of extremal functions." Communications on Pure and Applied Mathematics: A Journal Issued by the Courant Institute of Mathematical Sciences 54.2 (2001): 229-258.

\bibitem{dolbeault2012extremal} Dolbeault, Jean, and Maria J. Esteban. "Extremal functions for Caffarelli—Kohn—Nirenberg and logarithmic Hardy inequalities." Proceedings of the Royal Society of Edinburgh Section A: Mathematics 142.4 (2012): 745-767.

\bibitem{dong2024gagliardo} Dong, Mengxia. "Gagliardo–Nirenberg Inequality with Hölder Norms: M. Dong." Mediterranean Journal of Mathematics 22.6 (2025): 166.

\bibitem{evans2010partial} Evans, Lawrence C. Partial differential equations. Vol. 19. American Mathematical Society, 2022.

\bibitem{gagliardo1959ulteriori} Gagliardo, Emilio. “Ulteriori proprietà di alcune classi di funzioni in più variabili." Ricerche Mat. 8 (1959): 24-51.

\bibitem{ghoussoub2000multiple} Ghoussoub, Nassif, and C. Yuan. "Multiple solutions for quasi-linear PDEs involving the critical Sobolev and Hardy exponents." Transactions of the American Mathematical Society 352.12 (2000): 5703-5743.

\bibitem{john1961functions} John, Fritz, and Louis Nirenberg. "On functions of bounded mean oscillation." Communications on pure and applied Mathematics 14.3 (1961): 415-426.

\bibitem{kufner1995interpolation} Kufner, Alois, and Andreas Wannebo. “An interpolation inequality involving Hölder norms." (1995): 603-612.

\bibitem{lunardi2012analytic} Lunardi, Alessandra. Analytic semigroups and optimal regularity in parabolic problems. Springer Science \& Business Media, 2012.

\bibitem{lunardi2018interpolation} Lunardi, Alessandra. Interpolation theory. Vol. 16. Springer, 2018.

\bibitem{maz2013sobolev} Maz'ya, Vladimir. Sobolev spaces. Springer, 2013.

\bibitem{moser1971sharp} Moser, Jürgen. "A sharp form of an inequality by N. Trudinger." Indiana University Mathematics Journal 20.11 (1971): 1077-1092.

\bibitem{nirenberg2011elliptic} Nirenberg, Louis. “On elliptic partial differential equations." Annali della Scuola Normale Superiore di Pisa-Scienze Fisiche e Matematiche 13.2 (1959): 115-162.

\bibitem{opic1990hardy} Opic, Bohumír, and Alois Kufner. "Hardy-type inequalities." (No Title) (1990).

\bibitem{persson2017weighted} Persson, Lars-erik, Alois Kufner, and Natasha Samko. Weighted inequalities of Hardy type. World Scientific Publishing Company, 2017.

\bibitem{soudsky2018interpolation} Soudský, Filip, Anastasia Molchanova, and Tomáš Roskovec. “Interpolation between Hölder and Lebesgue spaces with applications." Journal of Mathematical Analysis and Applications 466.1 (2018): 160-168.

\bibitem{trudinger1967imbeddings} Trudinger, Neil S. "On imbeddings into Orlicz spaces and some applications." Journal of Mathematics and Mechanics 17.5 (1967): 473-483.

\bibitem{wang2000singular} Wang, Z. Q., and Michel Willem. "Singular minimization problems." Journal of Differential Equations 161.2 (2000): 307-320.


\end{thebibliography}

\end{document}